\documentclass[12pt,amsymb,fullpage]{amsart}
\usepackage{amssymb,amscd,pstricks}
\usepackage{fancyhdr}
\newtheorem{theorem}{Theorem}[section]
\newtheorem{defn}[theorem]{Definition}

\newtheorem{lemma}[theorem]{Lemma}

\newtheorem{eple}[theorem]{Example}
\newtheorem{rmk}[theorem]{Remarks}
\newtheorem{dsc}[theorem]{Discussion}
\newtheorem{nota}[theorem]{Notation}

\newsavebox{\indbin}
\savebox{\indbin}{\begin{picture}(0,0)
\newlength{\gnu}
\settowidth{\gnu}{$\smile$} \setlength{\unitlength}{.5\gnu}
\put(-1,-.65){$\smile$} \put(-.25,.1){$|$}
\end{picture}}

\newcommand{\be}{\begin{enumerate}}
\newcommand{\bd}{\begin{defn}}
\newcommand{\bt}{\begin{theorem}}
\newcommand{\bl}{\begin{lemma}}
\newcommand{\ee}{\end{enumerate}}
\newcommand{\ed}{\end{defn}}
\newcommand{\et}{\end{theorem}}
\newcommand{\el}{\end{lemma}}

\begin{document}
\title{A Simple Proof of the Uniform Convergence of\\
Fourier Series using Nonstandard Analysis}
\author{Tristram de Piro}
\address{Mathematics Department, The University of Exeter, Exeter}
 \email{tdpd201@exeter.ac.uk}
\thanks{}
\begin{abstract}
We give a proof of the uniform convergence of Fourier series, using the methods of nonstandard analysis.
\end{abstract}
\maketitle

We use the following notation;\\

\begin{defn}
\label{coefficient}
We let $\mathcal{S}$ denote the circle of radius ${1\over\pi}$, which we will model by the real interval $[-1,1]$, with the endpoints $\{-1,1\}$ identified. We denote by $C^{\infty}(\mathcal{S})$, the set of all functions $g:\mathcal{S}\rightarrow\mathcal{C}$, which are infinitely differentiable, in the sense of real manifolds. Equivalently, using L'Hospital's Rule, $C^{\infty}(\mathcal{S})$ consists of the functions $g:(-1,1)\rightarrow\mathcal{C}$, which are infinitely differentiable, and such that $g$ and all its derivatives $\{g^{(n)}:n\in\mathcal{N}\}$ extend to continuous functions on $\mathcal{S}$.\\

For such a function $g$, we define its $m$'th Fourier coefficient, where $m\in\mathcal{Z}$, by;\\

$\hat{g}(m)=\int_{\mathcal{S}}g(x)e^{-\pi ixm}dx=\int_{-1}^{1}g(x)e^{-\pi ixm}dx$\\
\end{defn}

\begin{theorem}{Uniform Convergence of Fourier Series}\\

\label{converge}

Let $g\in C^{\infty}(\mathcal{S})$, then;\\

$g(x)={1\over 2}\sum_{m\in\mathcal{Z}}\hat{g}(m)e^{\pi ixm}$ for all $x\in\mathcal{S}$\\

where the infinite sum is considered as $lim_{n\rightarrow\infty}S_{n}(g)$, $n\in\mathcal{N}$, for the finite sum $S_{n}(g)=\sum_{m=-n}^{m=n}\hat{g}(m)e^{\pi ixm}$. Moreover, the convergence is uniform on $\mathcal{S}$.
\end{theorem}

\begin{rmk}
\label{normal}
The result of Theorem \ref{converge} generalises immediately, to obtain, if $g\in C^{\infty}([a,b])$, (with $L=b-a>0$ and the obvious extension of the above definition), that;\\

$g(x)=\sum_{m\in\mathcal{Z}}\hat{g}(m)e^{2\pi i{xm\over L}}$ $(uniform\  convergence)$\\

where, for $m\in\mathcal{Z}$;\\

$\hat{g}(m)={1\over L}\int_{a}^{b}g(x)e^{-2\pi i{xm\over L}}dx$\\

This is achieved by a simple change of variables. Observe that this result also demonstrates that the functions $\{e_{n}(x)=e^{2\pi i{xm\over L}}:m\in\mathcal{Z}\}$ form an orthonormal "basis" of $C^{\infty}([a,b])\subset L^{2}([a,b])$, with respect to the (normalised by ${1\over L}$) inner product on $L^{2}([a,b])$. The more general result that they form an orthonormal basis of $L^{2}([a,b]$ follows from the fact that $C^{\infty}([a,b])$ is dense in $L^{2}([a,b])$. Proofs of these results (without using nonstandard analysis) can be found in \cite{SS} and \cite{Kor}.

\end{rmk}

In order to prove Theorem \ref{converge}, we follow the strategy of \cite{dep1}.
\begin{defn}
\label{spaces}
Let $\eta\in{{^{*}\mathcal{N}}\setminus\mathcal{N}}$, then we define;\\

${\overline{S}}_{\eta}=\{\tau\in{^{*}\mathcal{R}}:-1\leq\tau<1\}$\\

As in Definition 0.5 of \cite{dep1}, $\mathfrak{C}_{\eta}$ is the $*$-finite algebra consisting of internal unions of intervals of the form $[{i\over\eta},{i+1\over\eta})$, for $-\eta\leq i<\eta$, $i\in{^{*}\mathcal{N}}$. $\lambda_{\eta}$ is the counting measure on $\mathfrak{C}_{\eta}$ given by $\lambda_{\eta}([{i\over\eta},{i+1\over\eta}))={1\over\eta}$. Then $({\overline{S}}_{\eta},\mathfrak{C}_{\eta},\lambda_{\eta})$ is a hyperfinite measure space with $\lambda_{\eta}({\overline{S}}_{\eta})=2$. We denote by $({\overline{S}}_{\eta},L(\mathfrak{C}_{\eta}),L(\lambda_{\eta}))$, the associated Loeb space. The existence of such a space and the extension of ${^{\circ}\lambda}$ to $\sigma(\mathfrak{C}_{\eta})$ is shown in \cite{Loeb}. After producing the extension, we are then passing to the completion.\\

We let $([0,1],\mathfrak{B},\mu)$ denote the completion of the restriction of the Borel field on $\mathcal{R}$ to $[0,1]$, with respect to Lebesgue measure $\mu$. We let $(S,\mathfrak{B},\mu)$ denote the obvious corresponding measure space on $S$, so the identification map $G:[-1,1]\rightarrow S$ is continuous, measurable and measure preserving.
\end{defn}

\begin{lemma}
\label{stpart}
The standard part mapping;\\

\indent \\

$st:({\overline{S}}_{\eta},L(\mathfrak{C}_{\eta}),L(\lambda_{\eta}))\rightarrow ([-1,1],\mathfrak{B},\mu)$\\

is measurable and measure preserving. In particular, if $g\in C^{\infty}(S)$, then $(G\circ st)^{*}(g)$ is integrable with respect to $L(\lambda_{\eta})$ and;\\

$\int_{{\overline{S}}_{\eta}}(G\circ st)^{*}(g)dL(\lambda_{\eta})=\int_{-1}^{1}G^{*}(g) d\mu=\int_{S}g d\mu$\\

\end{lemma}

\begin{proof}
Using the proof of Theorem 0.7 in \cite{dep1}, or Theorem 14 of \cite{and}, and the remark in Definition \ref{spaces}.
\end{proof}

We recall Definition 0.8, Theorem 0.9, Definition 0.10, Theorem 0.11 and Lemma 0.12 from \cite{dep1}.

\begin{defn}
\label{disccoeff}
Let $n\in\mathcal{N}_{>0}$, let $\mathfrak{G}$, $G_{n,n}$ be as in Definition 0.8 of \cite{dep1}, and $g\in L^{1}(G_{n,n})$, with respect to the probability measure $\mu_{G}$. Let $\lambda_{G}$ be the rescaled measure, given by $\lambda_{G}=2\mu_{G}$. For $m\in\mathcal{N}$, with $-n\leq m\leq n-1$, we define the $m$'th discrete Fourier coefficient $\hat{g}(m)\in\mathcal{C}$ by;\\

$\hat{g}(m)=\int_{G_{n,n}}g(x)exp(-\pi ixm)d\lambda_{G}$ $(x\in G_{n,n})$\\

\end{defn}

We then have;\\

\begin{theorem}{Inversion Theorem for $G_{n,n}$}\\
\label{discinversion}

Let $\{G_{n,n},\lambda_{G},g,\hat{g}(m)\}$ be as in Definition \ref{disccoeff}, then;\\

$g(x)={1\over 2}\sum_{-n\leq m\leq n-1}\hat{g}(m)exp(\pi ixm)$\\

\end{theorem}

\begin{proof}
By Lemma 0.12 of \cite{dep1}, the characters on $G_{n,n}$ are given by;\\

$\gamma_{y}(x)=exp({\pi in^{2}\over n}xy)=exp(\pi inxy)$ $(*)$\\

for $x,y\in G_{n,n}$. Using Definition 0.10 of \cite{dep1}, and the fact that $\mu_{G}(x)={1\over 2n}$, for $x\in G_{n,n}$, we have;\\

$\hat{g}(\gamma_{y})={1\over 2n}\sum_{w\in G_{n,n}}g(w)exp(-\pi inwy)$ $(**)$\\

By Theorem 0.11 of \cite{dep1}, $(*)$, $(**)$ and the fact that $\lambda_{G}(x)={1\over n}$, for $x\in G_{n,n}$;\\

$g(x)=\sum_{y\in G_{n,n}}\hat{g}(\gamma_{y})\gamma_{y}(x)$\\

$=\sum_{y\in G_{n,n}}\hat{g}(\gamma_{y})exp(\pi inxy)$\\

$={1\over 2}\sum_{y\in G_{n,n}}({1\over n}\sum_{w\in G_{n,n}}g(w)exp(-\pi inwy))exp(\pi inxy)$\\

$={1\over 2}\sum_{-n\leq m\leq n-1}({1\over n}\sum_{w\in G_{n,n}}g(w)exp(-\pi iwm))exp(\pi ixm)$ $(y={m\over n})$\\

$={1\over 2}\sum_{-n\leq m\leq n-1}\hat{g}(m)exp(\pi ixm)$\\

\end{proof}

\begin{defn}
\label{nscoeff}
We adopt similar notation to Definition 0.15 of \cite{dep1}. $({\overline{\mathcal{S}}}_{\eta},\mathfrak{C}_{\eta},\lambda_{\eta})$ are as in Definition \ref{spaces}. We define;\\

$\overline{\mathcal{Z}}_{\eta}=\{m\in{^{*}{\mathcal{N}}}:-\eta\leq m\leq\eta-1\}$\\

$\mathfrak{D}_{\eta}$ will denote the ${^{*}}$-finite algebra on $\overline{\mathcal{Z}}_{\eta}$, consisting of internal subsets, with counting measure ${\delta}_{\eta}$, defined by $\delta_{\eta}(m)=1$. We let $\mathfrak{E}_{\eta}$ denote the ${^{*}}$-finite algebra on ${\overline{{\mathcal S}_{\eta}}}\times{\overline{{\mathcal Z}_{\eta}}}$, consisting of internal unions of the form $[{i\over\eta},{i+1\over\eta})\times m$, $-\eta\leq i,m\leq \eta-1$,  and let $\epsilon_{\eta}$ be the counting measure $\lambda_{\eta}\times\delta_{\eta}$.\\

We let ${^{*}exp(\pi i xt)},{^{*}exp(-\pi i xt)}:{^{*}{\mathcal{S}}}\times{^{*}{\mathcal{Z}}}\rightarrow{^{*}{\mathcal C}}$ be the transfers of the functions $exp(\pi ixt),exp(-\pi ixt):{\mathcal S}\times\mathcal{Z}\rightarrow{\mathcal C}$, and use the same notation to denote the restrictions of the transfers to ${\overline{{\mathcal S}_{\eta}}}\times\overline{\mathcal{Z}}_{\eta}$.\\

We let $exp_{\eta}(\pi ixt),exp_{\eta}(-\pi ixt):{\overline{{\mathcal S}_{\eta}}}\times{\overline{{\mathcal Z}_{\eta}}}\rightarrow{^{*}{\mathcal C}}$ denote their $\mathfrak{E}_{\eta}$-measurable counterparts, defined by;\\

$exp_{\eta}(\pi ixt)={^{*}exp(\pi i{[\eta x]\over\eta}t)}$, $(x,t)\in{\overline{{\mathcal S}_{\eta}}}\times{\overline{{\mathcal Z}_{\eta}}}$\\

and, similarly, for $exp_{\eta}(-\pi ixt)$. Given $f:\overline{{\mathcal S}_{\eta}}\rightarrow{^{*}\mathcal{C}}$, which is $\mathfrak{C}_{\eta}$-measurable, we define the nonstandard $m$'th Fourier coefficient of $f$, for $m\in\overline{\mathcal{Z}}_{\eta}$, by;\\

$\hat{f}_{\eta}(m)=\int_{\overline{{\mathcal S}_{\eta}}}f(x)exp_{\eta}(-\pi ixm) d\lambda_{\eta}$\\

so $\hat{f}_{\eta}:\overline{{\mathcal Z}_{\eta}}\rightarrow{^{*}\mathcal{C}}$ is $\mathfrak{D}_{\eta}$-measurable. $(*)$\\

Given $g:[-1,1]\rightarrow{\mathcal C}$, we let ${^{*}g}:{^{*}[-1,1]}\rightarrow{^{*}{\mathcal{C}}}$ denote its transfer and its restriction to $\overline{{\mathcal S}_{\eta}}$. We let $g_{\eta}$ denote its $\mathfrak{C}_{\eta}$-measurable counterpart, as above, and let $\hat{g}_{\eta}$ be as in $(*)$.\\

For $n\in\mathcal{N}$, we let $\mathcal{S}_{n}=[-1,1)$. We let $\mathfrak{C}_{n,st}$ consist of all finite unions of intervals of the form $[{i\over n},{i+1\over n})$, for $-n\leq i\leq n-1$. $\lambda_{n,st}$ is defined on $\mathfrak{C}_{n,st}$, by setting $\lambda_{n}([{i\over n},{i+1\over n}))={1\over n}$. We let;\\

$\mathcal{Z}_{n}=\{m\in\mathcal{N}:-n\leq m\leq n-1\}$\\

$\{\mathfrak{D}_{n,st},\delta_{n,st},\epsilon_{n,st}, exp_{n,st}(\pi i xt), exp_{n,st}(-\pi i xt)\}$ are all defined as above, restricting to $[-1,1)$ and $\mathcal{Z}$. If $g:[-1,1]\rightarrow{\mathcal C}$, we similarly define, $\{g_{n,st},\hat{g}_{n,st}\}$, ($st$ is suggestive notation for standard). Observe that $\lambda_{n,st}$ is just the restriction of Lebesgue measure $\mu$ to $\mathfrak{C}_{n,st}$, and transfers to $\lambda_{n}$.\\

$\{exp_{n,st}(\pi i xt), exp_{n,st}(-\pi i xt),g_{n,st},\hat{g}_{n,st}\}$ are all standard functions, which transfer to $\{exp_{n}(\pi i xt), exp_{n}(-\pi i xt),g_{n},\hat{g}_{n}\}$.\\
\end{defn}

We now obtain;\\

\begin{lemma}{Inversion Theorem for ${\overline{\mathcal{S}}}_{\eta}$}\\

\label{nsinversion}
Let $\{{\overline{\mathcal{S}}}_{\eta},{\overline{\mathcal{Z}}}_{\eta},f,\hat{f}_{\eta}\}$ be as in Definition \ref{nscoeff}, then;\\

$f(x)={1\over 2}\sum_{m\in{\overline{\mathcal{Z}}}_{\eta}}\hat{f}_{\eta}(m)exp_{\eta}(\pi ixm)$ $(x\in{\overline{\mathcal{S}}}_{\eta})$\\
\end{lemma}

\begin{proof}
As $f(x)$ and $exp_{\eta}(\pi ixm)$ are both $\mathfrak{C}_{\eta}$-measurable, the two sides of the equation are unchanged if we replace $x$ by ${[\eta x]\over\eta}$. Now the result follows, by transfer, from Theorem \ref{discinversion}, the definition of the internal integral $\int_{\overline{\mathcal{S}}_{\eta}}$ on $\overline{\mathcal{S}}_{\eta}$, see Definition 1.3 of \cite{dep1}, and the definition of a hyperfinite sum $\sum_{m\in{\overline{\mathcal{Z}}}_{\eta}}$ on $\overline{\mathcal{Z}}_{\eta}$, see Definition 2.19 of \cite{dep2}. Again, the reader should consider footnote 5 of \cite{dep1}.
\end{proof}

We now specialise the result of Lemma \ref{nsinversion} to $(\overline{\mathcal{S}}_{\eta},L(\mathfrak{C}_{\eta}),L(\lambda_{\eta}))$ and $(\overline{\mathcal{Z}}_{\eta},L(\mathfrak{D}_{\eta}),L(\delta_{\eta}))$. As in \cite{dep1}, the problem is to obtain the $S$-integrability conditions.\\

\begin{theorem}
\label{gSinteg}
Let $g\in C^{\infty}(S)$, and let $g$ also denote the pullback $\Gamma^{*}(g)\in C^{\infty}([0,1])$, then $g_{\eta}$, as given in Definition \ref{nscoeff}, is $S$-integrable on $\overline{\mathcal{S}}_{\eta}$. Moreover, ${^{\circ}g_{\eta}}=st^{*}(g)$, everywhere $L(\lambda_{\eta})$, and;\\

${^{\circ}\int_{\overline{{\mathcal S}_{\eta}}}g_{\eta} d\lambda_{\eta}}=\int_{\overline{{\mathcal S}_{\eta}}}st^{*}(g) d L(\lambda_{\eta})=\int_{S} g d\mu$\\
\end{theorem}

\begin{proof}
The proof is essentially contained in Theorem 0.17 of \cite{dep1}, but we give it, here, for the convenience of the reader. As $g$ is continuous on the interval $[-1,1)$, by Darboux's Theorem, see \cite{bs}, there exists $M\in\mathcal{N}$, such that for all $n\geq M$;\\

$|\int_{-1}^{1}g d\mu-\int_{-1}^{1}g_{n,st}d\lambda_{n,st}|<\epsilon$ $(*)$\\

Transferring the result $(*)$, see \cite{dep1} (Theorem 0.17), and using Lemma \ref{stpart}, we obtain;\\

$\int_{\overline{{\mathcal S}_{\eta}}}g_{\eta} d\lambda_{\eta}\simeq \int_{S} g d\mu=\int_{\overline{{\mathcal S}_{\eta}}} st^{*}(g) dL(\lambda_{\eta})$ $(**)$\\

As $g$ is continuous on $[-1,1]$, by \cite{dep1} (Theorem 1.6), we have $g_{\eta}(x)={^{*}g}({[\eta x]\over\eta})\simeq g({^{\circ}x})$, for all $x\in\overline{\mathcal{S}}_{\eta}$. Hence, ${^{\circ}g_{\eta}}=st^{*}(g)$ on $\overline{\mathcal{S}}_{\eta}$, $(***)$. We have that $g_{\eta}$ is $\mathfrak{C}_{\eta}$-measurable, and by $(**),(***)$, ${^{\circ}g_{\eta}}$ is integrable $L(\lambda_{\eta})$ and;\\

${^{\circ}\int_{\overline{{\mathcal S}_{\eta}}}g_{\eta} d\lambda_{\eta}}=\int_{\overline{{\mathcal S}_{\eta}}}{^{\circ}g_{\eta}}d L(\lambda_{\eta})$ $(\dag)$\\

Using Remarks 3.21 of \cite{dep2}, it follows that $g_{\eta}$ is $S$-integrable, and, clearly, the rest of the Theorem follows from $(**),(***),(\dag)$.

\end{proof}

We now show a corresponding result for $\hat{g}_{\eta}$. We require the following, observe that the definitions are slightly adjusted from Definition 0.18 of \cite{dep1}.

\begin{defn}
\label{discretederivshift}
If $n\in\mathcal{N}$, and $g_{n,st}$ is $\mathfrak{C}_{n,st}$-measurable, we define the discrete derivative $g'_{n,st}$ by;\\

$g'_{n,st}({j\over n})=n(g_{n,st}({j+1\over n})-g_{n,st}({j\over n}))$ $(-n\leq j<n-1)$\\

$g'_{n,st}({n-1\over n})=0$\\

$g'_{n,st}(x)=g'_{n,st}({[nx]\over n})$   $(x\in\mathcal{S}_{n})$\\

and the shift $g^{sh}_{n,st}$ by;\\

$g^{sh}_{n,st}({j\over n})=g_{n,st}({j+1\over n})$  $(-n\leq j<n-1)$\\

$g^{sh}_{n,st}({n-1\over n})=0$\\

$g^{sh}_{n,st}(x)=g^{sh}_{n,st}({[nx]\over n})$   $(x\in\mathcal{S}_{n})$\\

So both are $\mathfrak{C}_{n,st}$-measurable.

\end{defn}

\begin{lemma}{Discrete Calculus Lemmas}\\
\label{disccalc}

Let $\{g_{n,st},h_{n,st}\}$ be $\mathfrak{C}_{n,st}$-measurable and let $\{g'_{n,st},h'_{n,st},g^{sh}_{n,st},h^{sh}_{n,st}\}$ be as in Definition \ref{discretederivshift}. Then;\\

$(i)$. $\int_{\mathcal{S}_{n}}g'_{n,st} d\lambda_{n,st}=g_{n,st}({n-1\over n})-g_{n,st}(-1)$\\

$(ii)$. $(g_{n,st}h_{n,st})'=g'_{n,st}h^{sh}_{n,st}+g_{n,st}h'_{n,st}$\\

$(iii)$. $\int_{\mathcal{S}_{n}}g'_{n,st}h_{n,st}d\lambda_{n,st}=-\int_{\mathcal{S}_{n}}g^{sh}_{n,st}h'_{n,st}d\lambda_{n,st}+gh_{n,st}({n-1\over n})-gh_{n,st}(-1)$\\
\end{lemma}
\begin{proof}
Just use Lemma 0.19 of \cite{dep1}, with $n$ replacing $n^{2}$ in the proof.

\end{proof}

\begin{defn}
\label{functiontails}

For $n\in\mathcal{N}$, we let $\phi_{n}:\mathcal{N}\rightarrow\mathcal{C}$ be defined by $\phi_{n}(m)=n(exp({-\pi im\over n})-1)$, and let $\psi_{n}:\mathcal{N}\rightarrow\mathcal{C}$ be defined by $\psi_{n}(m)=n(exp({\pi im\over n})-1)$. Then, restricting to $\mathcal{Z}_{n}$, $\{\phi_{n},\psi_{n}\}$ are $\mathfrak{D}_{n,st}$-measurable on $\mathcal{Z}_{n}$.
If $g_{n,st}$ is $\mathfrak{C}_{n,st}$-measurable, for $-n\leq m\leq n-1$, we let;\\

$C_{n}(m)=g_{n,st}({n-1\over n})exp_{n,st}(-\pi i{n-1\over n}m)-g_{n,st}(-1)exp_{n,st}(-\pi i(-1)m)$\\

$D_{n}(m)=-{1\over n}g_{n,st}(-1)exp_{n,st}(\pi i{m\over n})exp_{n,st}(-\pi i(-1)m)$.\\

$C'_{n}(m)=-g'_{n,st}(-1)exp_{n,st}(-\pi i(-1)m)$\\

$D'_{n}(m)=-{1\over n}g'_{n,st}(-1)exp_{n,st}(\pi i{m\over n})exp_{n,st}(-\pi i(-1)m)$.\\

$E_{n}(m)=\phi_{n}(m)D_{n}(m)-C_{n}(m)$\\

$E'_{n}(m)=\phi_{n}(m)D'_{n}(m)-C'_{n}(m)$\\

$F_{n}(m)=\psi_{n}(m)\phi_{n}(m)D_{n}(m)-\psi_{n}(m)C_{n}(m)+\phi_{n}(m)D'_{n}(m)-C'_{n}(m)$\\

considered as $\mathfrak{D}_{n,st}$-measurable functions.
\end{defn}

\begin{lemma}{Discrete Fourier transform}\\
\label{dft}

Let $g_{n,st}$ be $\mathfrak{C}_{n,st}$-measurable. Then, for $m\neq 0$;\\

$\hat{g}_{n,st}(m)={\hat{g'}_{n,st}(m)+E_{n}(m)\over\psi_{n}(m)}={\hat{g''}_{n,st}(m)+F_{n}(m)\over\psi_{n}^{2}(m)}$\\

\end{lemma}

\begin{proof}
Again, use Lemma 0.21 of \cite{dep1}, with the simple adjustment of replacing $n^{2}$ by $n$ in the proof.
\end{proof}

\begin{lemma}
\label{unifbounded}
If $g\in C^{\infty}([-1,1])$, with $g(-1)=g(1)=0$, then the functions $\hat{g''}_{n,st}(m)$ and $F_{n}(m)$ are uniformly bounded, independently of $n$, for $n\geq 1$.
\end{lemma}

\begin{proof}

The proof is similar to Lemma 0.22 of \cite{dep1}, with some minor modifications. Observing that;\\

 $D_{n}(m)=0$, $(g(-1)=0)$\\

 $|D'_{n}(m)|\leq {1\over n}|g'_{n,st}|(-1)$\\

 $|\phi_{n}(m)|\leq 2n$, $|\psi_{n}(m)|\leq 2n$\\

$|C_{n}(m)|\leq |g_{n,st}|({n-1\over n})$, $(g(-1)=0)$\\

$|C'_{n}(m)|\leq|g'_{n,st}|(-1)$\\

we obtain;\\

$|F_{n}(m)|\leq 2n|g_{n,st}|({n-1\over n})+3n|g'_{n,st}|(-1)$\\

$=2n|g_{n,st}|({n-1\over n})+3n|g_{n,st}|({1-n\over n})$, $(g(-1)=0)$\\

Now assuming that $g$ is real valued, otherwise take real and imaginary parts, we can apply the mean value theorem, and using the assumptions on $g$, we have that;\\

$-ng_{n,st}({n-1\over n})=g'(c_{n})$, $c_{n}\in ({n-1\over n},1)$\\

$ng_{n,st}({1-n\over n})=g'(d_{n})$, $d_{n}\in (-1,{1-n\over n})$\\

$|F_{n}(m)|\leq 2|g'(c_{n})|+3|g'(d_{n})|\leq 5D$\\

where $D=||g'||_{C(\mathcal{S})}$.\\

We now follow through the rest of the proof of Lemma 0.22 in \cite{dep1}, replacing $n^{2}$ by $n$, to obtain;\\

$|\hat{g''}_{n,st}(m)|\leq M+2B$\\

where $M=||g''||_{L^{1}(\mathcal{S})}$, and $B=||g||_{C(\mathcal{S})}$.

\end{proof}

\begin{lemma}
\label{tailvanish}
If $g\in C^{\infty}([-1,1])$, there exists a constant $H\in\mathcal{R}$ such that, for all $n\geq 1$, and $-n\leq m\leq n-1$, $m\neq 0$, with $n\in\mathcal{N}$, $m\in\mathcal{Z}$;\\

$|\hat{g}_{n,st}(m)|\leq {H\over m^{2}}$\\

Moreover, if $\epsilon>0$ is standard, there exists a constant $N(\epsilon)\in\mathcal{N}_{>0}$, such that for all $n>N(\epsilon)$, for all $L,L'\in\mathcal{N}$ with $N(\epsilon)<|L|\leq |L'|\leq n$, $LL'>0$;\\

$\int_{L}^{L'}|\hat{g}_{n,st}|(m)d\delta_{n}(m)<\epsilon$\\

\end{lemma}

\begin{proof}

As in Lemma 0.23 of \cite{dep1}, we have;\\

$|\psi_{n}(m)|^{2}\geq 4m^{2}$ $(-n\leq m\leq n-1)$ $(*)$\\

The function $h=g-c$, where $c=g(0)=g(1)$, satisfies the hypotheses of Lemma \ref{unifbounded}. Let $W$ be the constant bound obtained there. Then, using Lemma \ref{dft}, $(*)$, we obtain, for $m\neq 0$;\\

$|\hat{h}_{n,st}(m)|\leq {W\over 4m^{2}}$ $(**)$\\

Now observe that $\hat{g}_{n,st}=\hat{h}_{n,st}+\hat{c}_{n,st}$ and, using, for example, Lemma \ref{discinversion}, that $\hat{c}_{n,st}(m)=0$, for $m\neq 0$, $(***)$. Then, combining $(**),(***)$, we obtain the first result with $H={W\over 4}$. Now;\\

$\int_{L}^{L'}|\hat{g}|_{n,st}(m)d\delta_{n}(m)$\\

$\leq\int_{L}^{n}|\hat{g}|_{n,st}(m)d\delta_{n}(m)$\\

$\leq\int_{L}^{n}{H\over m^{2}}d\delta_{n}(m)$\\

$=\sum_{k=L}^{n-1}{H\over m^{2}}$\\

$\leq\int_{L-1}^{n-1}{H\over x^{2}}dx$\\

$=[{-H\over x}]_{L-1}^{n-1}={H\over L-1}-{H\over n-1}<\epsilon$\\

if $min(n,L)>N(\epsilon)={2H\over\epsilon}+1$\\

\end{proof}

We can now show the analogous result to Theorem \ref{gSinteg}. We require some further notation;\\

\begin{defn}
\label{zext}
If $g\in C^{\infty}([-1,1])$, with Fourier coefficients $\hat{g}(m)$, for $m\in\mathcal{Z}$, as given in \ref{coefficient}, then we consider $\hat{g}:\mathcal{Z}\rightarrow\mathcal{C}$, as a measurable function on $(\mathcal{Z},\mathfrak{D},\delta)$, where $\mathfrak{D}$ is the $\sigma$-algebra of subsets of $\mathcal{Z}$, and $\delta$ is the counting measure, given by $\delta(m)=1$, for $m\in\mathcal{Z}$. We let $\mathcal{Z}^{+-\infty}$ denote the extended integers $\mathcal{Z}\cup\{+\infty,-\infty\}$. We let $\hat{g}_{\infty}$ be the extension of $\hat{g}$ to $\mathcal{Z}^{+-\infty}$, obtained by setting $\hat{g}_{\infty}(+\infty)=\hat{g}_{\infty}(-\infty)=0$, (\footnote{\label{zmmp} As in Lemma 0.6 of \cite{dep1}, it is easy to show there exists a unique $\sigma$-algebra $\mathfrak{D}'$ on $\mathcal{Z}^{+-\infty}$, which separates the points $+\infty$ and $-\infty$, and such that $\mathfrak{D}'|_{\mathcal{Z}}=\mathfrak{D}$. Moreover, there is a unique extension of $\delta$ to a complete measure $\delta'$ on $\mathfrak{D}'$, with $\delta'(+\infty)=\delta'(-\infty)=\infty$. As in Theorem 0.7 of \cite{dep1}, it is straightforward to show that;\\

$st:(\overline{\mathcal{Z}}_{\eta},L(\mathfrak{D}_{\eta}),L(\delta_{\eta}))\rightarrow (\mathcal{Z}^{+-\infty},\mathfrak{D}',\delta')$\\

is measurable and measure preserving. In particular, if $\hat{g}$ is integrable $\delta$ (iff $st^{*}(\hat{g}_{\infty})$ is integrable $L(\delta_{\eta})$), we have;\\

$\int_{\overline{\mathcal{Z}}_{\eta}}st^{*}(\hat{g}_{\infty})dL(\delta_{\eta})=\int_{\mathcal{Z}^{+-\infty}}\hat{g}_{\infty}d\delta'
=\int_{\mathcal{Z}}\hat{g}d\delta$\\}).
\end{defn}

\begin{theorem}
\label{transgSinteg}

Let $g\in C^{\infty}([-1,1])$, then $\hat{g}_{\eta}$, as given in Definition \ref{nscoeff}, is $S$-integrable on $\overline{{\mathcal Z}_{\eta}}$. Moreover ${^{\circ}\hat{g}}_{\eta}=st^{*}(\hat{g}_{\infty})$, everywhere $L(\delta_{\eta})$, and;\\

${^{\circ}\int_{\overline{{\mathcal Z}_{\eta}}}\hat{g}_{\eta} d\delta_{\eta}}=\int_{\overline{{\mathcal Z}_{\eta}}}st^{*}(\hat{g}_{\infty}) d L(\delta_{\eta})=\int_{\mathcal{Z}} \hat{g} d\delta$\\

see Definition \ref{zext} and footnote \ref{zmmp} for relevant terminology.
\end{theorem}

\begin{proof}
By Lemma \ref{tailvanish};\\

$\mathcal{R}\models (\forall n_{(n>N(\epsilon))})(\forall L,N_{(LN\geq 0, N(\epsilon)<|L|,|N|<n)})\int_{L}^{N}|\hat{g}_{n,st}| d\delta_{n,st} <\epsilon$\\

Hence, the corresponding statement is true in ${^{*}\mathcal{R}}$. In particular, if $\eta$ is infinite, and $\{L,N\}$ are infinite, of the same sign, belonging to $\overline{\mathcal{Z}}_{\eta}$, we have that;\\

$\int_{L}^{N}|\hat{g}_{\eta}| d\delta_{\eta} < \epsilon$\\

As $\epsilon$ was arbitrary we conclude that;\\

$\int_{L}^{N}|\hat{g}_{\eta}| d\delta_{\eta}\simeq 0$ $(*)$\\

for all infinite $\{L,N\}$, of the same sign, in $\overline{\mathcal{R}}_{\eta}$. Now, using Definition \ref{nscoeff} and the fact that $|exp_{\eta}(-\pi ixm)|\leq 1$, by transfer, we have, for $m\in\overline{\mathcal{Z}}_{\eta}$;\\

$|\hat{g}_{\eta}(m)|\leq\int_{\overline{{\mathcal S}_{\eta}}}|g_{\eta}(x)d\lambda_{\eta}=C$\\

where $C$ is finite, as, by Theorem \ref{gSinteg}, $g_{\eta}$ is $S$-integrable. It follows that for $n\in\mathcal{N}$, the functions $\hat{g}_{\eta}\chi_{[-n,n]}$ are finite, in the sense of Definition 1.7 of \cite{dep1}. Now, proceeding as in Theorem 0.17 of \cite{dep1}, replacing $\overline{\mathcal{R}}_{\eta}$ by $\overline{\mathcal{Z}}_{\eta}$, we obtain that $\hat{g}_{\eta}$ is $S$-integrable.\\

 If $m\in{\mathcal{Z}_{\eta}}$, the function $r_{m}(x)=g_{\eta}(x)exp_{\eta}(-\pi ixm)$ is $S$-integrable, by Corollary 5 of \cite{and}, as $|r_{m}|\leq |g_{\eta}|$, and $g_{\eta}$ is $S$-integrable, by Theorem \ref{gSinteg}. Then, if $m\in\mathcal{Z}_{\eta}$ is finite, as in Theorem 0.24 of \cite{dep1}, just replacing ${\overline{R}_{\eta}}$ by ${\overline{S}_{\eta}}$, and using Definition \ref{nscoeff}, Theorem 1.9 of \cite{dep1}, Theorem \ref{gSinteg}, continuity of exp, see Theorem 1.6 of \cite{dep1}, and Lemma \ref{stpart};\\

${^{\circ}\hat{g}_{\eta}}(m)=\hat{g}(m)=st^{*}(\hat{g}_{\infty})(m)$ $(**)$\\

Now, using the first part of Lemma \ref{tailvanish}, we obtain, by transfer, that for infinite $m\in\overline{\mathcal{Z}}_{\eta}$, $\hat{g}_{\eta}(m)\simeq 0$. By Definition \ref{zext}, we have $st^{*}(\hat{g}_{\infty})(m)=0$, $(***)$. Combining $(**),(***)$ gives ${^{\circ}\hat{g}_{\eta}}=st^{*}(\hat{g}_{\infty})$, everywhere $L(\delta_{\eta})$. The rest of the theorem follows from footnote \ref{zmmp}.
\end{proof}

Finally, we have;\\

\begin{theorem}
\label{nspfit}
For $g\in C^{\infty}(\mathcal{S})$, there is a non standard proof of the uniform convergence of its Fourier series, Theorem \ref{converge}.

\end{theorem}
\begin{proof}
By Lemma \ref{nsinversion}, and using the definition of the internal integral $\int_{\mathcal{Z}_{\eta}}$ to replace the hyperfinite sum $\Sigma_{m\in\mathcal{Z}_{\eta}}$, we have that;\\

$g_{\eta}(x)={1\over 2}\int_{\overline{\mathcal Z}_{\eta}}\hat{g}_{\eta}(m)exp_{\eta}(\pi ixm)d\delta_{\eta}(m)$ $(*)$\\

for $x\in\overline{{\mathcal S}_{\eta}}$. As in Theorem \ref{transgSinteg}, the function $s_{x}(m)=\hat{g}_{\eta}(m)exp_{\eta}(\pi ixm)$ is $S$-integrable, because, by the same theorem, $\hat{g}_{\eta}$ is $S$-integrable. We now argue as before, and use the result that ${^{\circ}{g}}_{\eta}=st^{*}(\hat{g}_{\infty})$, everywhere $L(\delta_{\eta})$ $(**)$. Note that, as $\hat{g}_{\eta}$ is $S$-integrable, using Theorem 3.24 of \cite{dep2} and $(**)$, $st^{*}(\hat{g}_{\infty})$ is integrable $L(\delta_{\eta})$ and, therefore, by footnote \ref{zmmp} $\hat{g}$ is integrable $\delta$, $(***)$. We have, if $x\in [-1,1)$, taking standard parts in $(*)$;\\

$g(x)={^{\circ}g_{\eta}}(x)={1\over 2}\int_{\overline{\mathcal Z}_{\eta}}{^{\circ}\hat{g}_{\eta}}(m){^{\circ}exp_{\eta}}(\pi ixm)d L(\delta_{\eta})(m)$\\

$={1\over 2}\int_{m finite}st^{*}(\hat{g}_{\infty})(m)exp_{\eta}(\pi i{^{\circ}x}m)dL(\delta_{\eta})(m)$\\

$={1\over 2}\int_{m finite}st^{*}(\hat{g}_{\infty}exp_{\pi ix})(m)d L(\delta_{\eta})(m)$\\

$={1\over 2}\int_{\mathcal{Z}}\hat{g}(m)exp(\pi ixm)d\delta(m)$\\

$={1\over 2}\sum_{m\in\mathcal{Z}}\hat{g}(m)exp(\pi ixm)$ $(\dag)$, (\footnote{\label{caseone} The case $x=1$ obviously then follows from the fact that $g(1)=g(-1)$ and $exp(\pi im)=exp(-\pi im)$, for $m\in\mathcal{Z}$.})\\

The sum in $(\dag)$ can be taken as $lim_{n\rightarrow\infty}S_{n}(g)$, for\\
$S_{n}(g)=\sum_{m=-n}^{n}\hat{g}(m)exp(\pi ixm)$, $(\dag\dag)$ because, for $n\in\mathcal{N}$, $g_{n,x}(m)=(\hat{g}exp_{\pi ix}\chi_{[-n,n]})(m)$, converges everywhere $\delta$ to $(\hat{g}exp_{\pi ix})(m)$, so $(\dag\dag)$ follows from $(***)$ and the DCT. In order to obtain uniform convergence in $x$, we use the fact, from Theorem \ref{transgSinteg}, that, for $m\in\mathcal{Z}$, $\hat{g}(m)={^{\circ}\hat{g}_{\eta}}(m)$ and, the first part of Lemma \ref{tailvanish}, which gives, by transfer, that $|\hat{g}_{\eta}(m)|\leq {H\over m^{2}}$, for $m\in\mathcal{Z}_{\neq 0}$. Combining these results gives that $|\hat{g}(m)|\leq {H\over m^{2}}$, for $m\neq 0$, so $|g_{m}(x)|=|\hat{g}(m)exp(\pi ixm)|\leq{H\over m^{2}}$, $m\neq 0$, uniformly in $x\in [-1,1]$, $(\dag\dag\dag)$. Applying Weierstrass $M$-test, see \cite{bs}, and the estimate $(\dag\dag\dag)$, gives the required uniform convergence of the sums $S_{n}(g)=\sum_{m=-n}^{n}g_{m}$.

\end{proof}

\end{document}